\documentclass{amsart}
\usepackage{amsmath, amsfonts, amsthm, amssymb}
\usepackage{comment}
\usepackage{enumerate}
\usepackage{graphicx}
\usepackage{float}
\usepackage{url}
\usepackage{MnSymbol,wasysym}
\usepackage{multicol}
\usepackage{caption}
\usepackage{subcaption}
\usepackage{array}
\renewcommand{\epsilon}{\varepsilon}
\graphicspath{ {./Images/} }

\numberwithin{equation}{section} 
\numberwithin{figure}{section}
\numberwithin{table}{section} 

\theoremstyle{definition}\newtheorem{defn}{Definition}
\numberwithin{defn}{section}
\theoremstyle{definition}
\numberwithin{eg}{section}
\theoremstyle{theorem}\newtheorem{prop}{Proposition}
\numberwithin{prop}{section}
\theoremstyle{definition}
\theoremstyle{definition}
\theoremstyle{definition}
\theoremstyle{definition}
\theoremstyle{definition}\newtheorem*{openqs}{Open Questions}
\theoremstyle{theorem}\newtheorem{thm}{Theorem}
\numberwithin{thm}{section}
\theoremstyle{theorem}
\theoremstyle{definition}
\theoremstyle{definition}\newtheorem*{notation}{Notation}
\theoremstyle{definition}
\theoremstyle{definition}
\theoremstyle{definition}
\theoremstyle{theorem}
\numberwithin{cor}{section}

\newcommand{\C}{\mathbb C}

\newcommand{\R}{\mathbb R}

\renewcommand{\phi}{\varphi}

\newcommand{\Tr}{\operatorname{Tr}}

\newcommand{\inv}{^{-1}}

\title[Explicit Bounds for the Pseudospectra of Matrices and Operators]{Explicit Bounds for the Pseudospectra of Various Classes of Matrices and Operators}
\author[Gong, Meyerson, Meza, Stoiciu, Ward]{Feixue Gong$^1$, Olivia Meyerson$^2$, Jeremy Meza$^3$, Mihai Stoiciu$^4$, Abigail Ward$^5$ \\ }
\thanks{  $^{1,2, 4}$Williams College $^3$Carnegie Mellon University $^5$The University of Chicago}
\date{\today}

\begin{document}

\begin{abstract}

We study the $\epsilon$-pseudospectra $\sigma_\epsilon(A)$ of square matrices $A \in \mathbb{C}^{N \times N}$. We give a complete characterization of the $\epsilon$-pseudospectrum of any $2 \times 2$ matrix and describe the asymptotic behavior (as $\epsilon \to 0$) of $\sigma_\epsilon(A)$ for any square matrix $A$. We also present explicit upper and lower bounds for the $\epsilon$-pseudospectra of bidiagonal matrices, as well as for finite rank operators.

\end{abstract}

\maketitle

\section{Introduction}

The pseudospectra of matrices and operators is an important mathematical object that has found applications in various areas of mathematics: linear algebra, functional analysis, numerical analysis, and differential equations. An overview of the main results on pseudospectra can be found in \cite{trefethen2005spectra}.

In this paper we describe the asymptotic behavior of the $\epsilon$-pseudospectrum of any $n \times n$ matrix. We apply this asymptotic bound and additionally provide explicit bounds on their $\epsilon$-pseudospectra to several classes of matrices and operators, including $2 \times 2$ matrices, bidiagonal matrices, and finite rank operators.

The paper is organized as follows: in Section \ref{sec: pseudospectra} we give the three standard equivalent definitions for the pseudospectrum and we present the ``classical" results on $\epsilon$-pseudospectra of normal and diagonalizable matrices (the Bauer-Fike theorems). Section \ref{sec: 2x2} contains a detailed analysis of the $\epsilon$-pseudospectrum of $2 \times 2$ matrices, including both the non-diagonalizable case (Subsection \ref{subsec: non-diag}) and the diagonalizable case (Subsection \ref{subsec: diag}). The asymptotic behavior (as $\epsilon \to 0$) of the $\epsilon$-pseudospectrum of any $n \times n$ matrix is described in Section \ref{sec: audit}, where we show (in Theorem \ref{thm: AUDiT}) that, for any square matrix $A$, the $\epsilon$-pseudospectrum converges, as $\epsilon \to 0$ to a union of disks. We apply the main result of Section \ref{sec: audit} to several classes of matrices: matrices with a simple eigenvalue, matrices with an eigenvalue with geometric multiplicity 1, $2 \times 2$ matrices, and Jordan blocks. 

Section \ref{sec: bidiag} is dedicated to the analysis of arbitrary periodic bidiagonal matrices $A$. We derive explicit formulas (in terms the coefficients of $A$) for the asymptotic radii, given by Theorem \ref{thm: AUDiT}, of the $\epsilon$-pseudospectrum of $A$, as $\epsilon \to 0$. In the last section (Section \ref{sec: finiterank}) we consider finite rank operators and show that the $\epsilon$-pseudospectrum of an operator of rank $m$ is at most as big as $C \epsilon^{\frac{1}{m}}$, as $\epsilon \to 0$.

\section{Pseudospectra} \label{sec: pseudospectra}

\subsection{Motivation and Definitions}

The concept of the spectrum of a matrix $A \in \C^{N \times N}$ provides a fundamental tool for understanding the behavior of $A$. As is well-known, a complex number $z \in \C$ is in the spectrum of $A$ (denoted $\sigma(A)$) whenever $zI-A$ (which we will denote as $z-A$) is not invertible, i.e., the characteristic polynomial of $A$ has $z$ as a root. As slightly perturbing the coefficients of $A$ will change the roots of the characteristic polynomial, the property of ``membership in the set of eigenvalues" is not well-suited for many  purposes, especially those in numerical analysis. We thus want to find a characterization of when a complex number is close to an eigenvalue, and we do this by considering the set of complex numbers $z$ such that $\|(z-A)\inv\|$ is large, where the norm here is the usual operator norm induced by the Euclidean norm, i.e.
\[\|A\| = \displaystyle \sup_{\|v\|=1} \|Av\|. \]
The motivation for considering this question comes from the observation that if $z_n$ is a sequence of complex numbers converging to an eigenvalue $\lambda$ of $A$, then $\| (z_n-A)\inv\| \to \infty$ as $n \to \infty$. We call the operator $(z-A)\inv$ the $\emph{resolvent}$ of $A$. The observation that the norm of the resolvent is large when $z$ is close to an eigenvalue of $A$ leads us to the first definition of the $\epsilon$-pseudospectrum of an operator.
\begin{defn}
Let $A \in \mathbb{C}^{N \times N}$, and let $\epsilon > 0$. The $\epsilon$-pseudospectrum of $A$ is the set of $z \in \mathbb{C}$ such that
\[ \| (z-A)^{-1} \| > 1/\epsilon \]
\end{defn}

Note that the boundary of the $\epsilon$-pseudospectrum is exactly the $1/\epsilon$ level curve of the function $z \mapsto \|(z-A)^{-1} \|$. Fig. \ref{resnorm} depicts the behavior of this function near the eigenvalues.

\begin{figure}[h]
\includegraphics[scale=0.5]{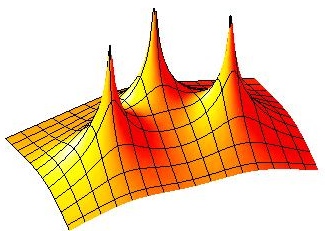}
\caption{Contour Plot of Resolvent Norm}
\label{resnorm}
\end{figure}

The resolvent norm has singularities in the complex plane, and as we approach these points, the resolvent norm grows to infinity. Conversely, if $\| (z-A)^{-1} \| $ approaches infinity, then $z$ must approach some eigenvalue of $A$ \cite[Thm 2.4]{trefethen2005spectra}.

(It is also possible to develop a theory of pseudospectrum for operators on Banach spaces, and it is important to note that this converse does not necessarily hold for such operators; that is, there are operators \cite{davies1999pseudo, davies1999semi} such that $\| (z-A)^{-1} \| $ approaches infinity, but $z$ does not approach the spectrum of $A$.)

The second and third definitions of the $\epsilon$-pseudospectrum arise from eigenvalue perturbation theory \cite{kato1995perturbation}.
\begin{defn}
Let $A \in \mathbb{C}^{N \times N}$. The $\epsilon$-pseudospectrum of $A$ is the set of $z \in \mathbb{C}$ such that
\[ z \in \sigma(A + E) \]
for some $E$ with $\| E \| < \epsilon$.
\end{defn}

\begin{defn}
Let $A \in \mathbb{C}^{N \times N}$. The $\epsilon$-pseudospectrum of $A$ is the set of $z \in \mathbb{C}$ such that
\[ \| (z-A)v \| < \epsilon \]
for some unit vector $v$.
\end{defn}
The third definition is similar to our first definition in that it quantifies how close $z$ is to an eigenvalue of $A$. In addition to this, it also gives us the notion of an $\epsilon$-\emph{pseudoeigenvector}.

\begin{thm}[Equivalence of the definitions of pseudospectra] For any matrix $A \in \mathbb{C}^{N \times N}$, the three definitions above are equivalent.  \end{thm}
The proof of this theorem is given in  \cite[\S 2]{trefethen2005spectra}. As all three definitions are equivalent, we can unambiguously denote the $\epsilon$-pseudospectrum of $A$ as $\sigma_\epsilon(A)$.

Fig.~\ref{nonnormal} depicts an example of $\epsilon$-pseudospectra for a specific matrix and for various $\epsilon$. We see that the boundaries of $\epsilon$-pseudospectra for a matrix are curves in the complex plane around the eigenvalues of the matrix. We are interested in understanding geometric and algebraic properties of these curves.

\begin{figure}[h]
\includegraphics[scale=0.5]{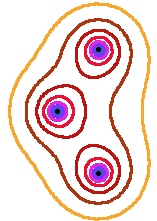}
\caption{The curves bounding the $\epsilon$-pseudospectra of a matrix $A$, for different values of $\epsilon$.}
\label{nonnormal}
\end{figure}

Several properties of pseudospectra are proven in \cite[\S 2]{trefethen2005spectra}. One of which is that if $A \in \C^{N \times N}$, then $\sigma_\epsilon(A)$ is nonempty, open, and bounded, with at most $N$ connected components, each containing one or more eigenvalues of $A$. This leads us to the following notation:
\begin{notation}
For $\lambda \in \sigma(A)$, we write $\sigma_\epsilon(A) \restriction \lambda$ to be the connected component of $\sigma_\epsilon(A)$ that contains $\lambda$.
\end{notation}
Another property, which follows straight from the definitions of pseudospectra, is that $\bigcap_{\epsilon > 0} \sigma_\epsilon(A) = \sigma(A)$. From these properties, it follows that there is $\epsilon$ small enough so that $\sigma_\epsilon(A)$ consists of exactly $|\sigma(A)|$ connected components, each an open set around a distinct eigenvalue. In particular, there is $\epsilon$ small enough so that $\sigma(A) \cap \sigma_\epsilon(A) \restriction \lambda = \{\lambda\}$.

When a matrix $A$ is the direct sum of smaller matrices, we can look at the pseudospectra of the smaller matrices to understand the $\epsilon$-pseudospectrum of $A$. We get the following theorem from \cite{trefethen2005spectra}:
\begin{thm}
 \[\sigma_\epsilon(A_1 \oplus A_2) = \sigma_\epsilon(A_1) \cup \sigma_\epsilon(A_2).\]
\label{directsum}
\end{thm}

\subsection{Normal Matrices}

 Recall that a matrix $A$ is \emph{normal} if $A A^\ast = A^\ast A$, or equivalently if $A$ can be diagonalized with an orthonormal basis of eigenvectors.

The pseudospectra of these matrices are particularly well-behaved: Thm.~\ref{normalpseudo} shows that the $\epsilon$-pseudospectrum of a normal matrix is exactly a disk of radius $\epsilon$ around each eigenvalue, as in shown in Fig.~\ref{normalex}. This is clear for diagonal matrices; it follows for normal matrices since as we shall see, the $\epsilon$-pseudospectrum  of a matrix is invariant under a unitary change of basis.

\begin{figure}[h]
\includegraphics[scale=0.5]{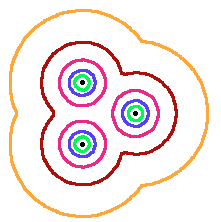}
\caption{The $\epsilon$-pseudospectrum of a normal matrix. Note that each boundary is a perfect disk around an eigenvalue.}
\label{normalex}
\end{figure}

\begin{thm} Let $A \in \mathbb{C}^{N \times N}$. Then,
\begin{equation} \sigma(A) + B(0, \epsilon) \subseteq \sigma_\epsilon(A)  \qquad \text{ for all } \; \epsilon > 0 \label{eq: normal1}. \end{equation}
Furthermore, $A$ is a normal matrix if and only if
\begin{equation} \sigma_\epsilon(A) = \sigma(A) + B(0,\epsilon)  \qquad \text{ for all } \; \epsilon > 0 \label{eq: normal2}. \end{equation}

\label{normalpseudo} \end{thm}

The proof of this theorem can be found in  \cite[\S 2]{trefethen2005spectra}

\subsection{Non-normal Diagonalizable Matrices}

Now suppose $A$ is diagonalizable but not normal, i.e. we cannot diagonalize $A$ by an isometry of $\C^N$. In this case we do not expect to get an exact characterization of the $\epsilon$-pseudospectra as we did previously. That is, there exist matrices with pseudospectra larger than the disk of radius $\epsilon$. Regardless, we can still characterize the behavior of non-normal, diagonalizable matrices.

\begin{thm}[Bauer-Fike] Let $A \in \mathbb{C}^{N \times N}$ and let A be diagonalizable, $A = V D V\inv$. Then for each $\epsilon > 0$,
\[\sigma(A) + B(0, \epsilon) \subseteq \sigma_\epsilon(A) \subseteq \sigma(A) + B(0, \epsilon \kappa(V)) \]
where
\[\kappa(V) = \| V \| \|V\inv \| = \frac{s_{\max}(V)}{s_{\min}(V)}, \]
and $s_{\max}(V)$, $s_{\min}(V)$ are the maximum and minimum singular values of $V$, respectively.
\label{bfv1} \end{thm}

Here, $\kappa(V)$ is known as the condition number of $V$. Note that $\kappa(V)\geq 1$, with equality attained if and only if $A$ is normal. Thus, $\kappa(V)$ can be thought of as a measure of the normality of a matrix. However, there is some ambiguity when we define $\kappa(V)$, as $V$ is not uniquely determined. If the eigenvalues are distinct, then $\kappa(V)$ becomes unique if the eigenvectors are normalized by $\|v_j\| = 1$.

%

\subsection{Non-diagonalizable Matrices}

So far we have considered normal matrices, and more generally diagonalizable matrices. We now relax our constraint that our matrix be diagonalizable, and provide similar bounds on the pseudospectra.
While not every matrix is diagonalizable, every matrix can be put in Jordan normal form. Below we give a brief review of the Jordan form.

Let $A \in \mathbb{C}^{N \times N}$ and suppose $A$ has only one eigenvalue, $\lambda$ with geometric multiplicity one. Writing $A$ in Jordan form, there exists a matrix $V$ such that $AV=VJ$,
where $J$ is a single Jordan block of size $N$. Write
\[V= \begin{pmatrix}
v_1, & v_2, & \ldots \ ,& v_n \end{pmatrix}\]
Then,
\[AV= \begin{pmatrix} Av_1, & Av_2, & \ldots \ , & Av_n \end{pmatrix} = \begin{pmatrix}  \lambda v_1, & v_1+\lambda v_2, & \ldots \ ,& v_{n-1} + \lambda v_n \end{pmatrix} = VJ, \]
and hence $v_1$ is a right eigenvector associated with $\lambda$ and $v_2, ..., v_n$ are \emph{generalized right eigenvectors}, that is right eigenvectors for $(A-\lambda I)^k$ for $k > 1$.
Similarly, there exists a matrix $U$ such that $U^\ast A = JU^\ast$, where now the rows of $U^\ast$ are left generalized eigenvectors associated with $\lambda$.

We can also quantify the normality of an eigenvalue in the same way $\kappa(V)$ quantifies the normality of a matrix.
\begin{defn} For any simple eigenvalue $\lambda_j$ of a matrix $A$, the \emph{condition number of $\lambda_j$} is defined as
\[\kappa(\lambda_j) = \frac {\| u_j\| \|v_j\| }{|u_j ^* v_j|}, \]
where $v_j$ and $u^*_j$ are the right and left eigenvectors associated with $\lambda_j$, respectively.
\end{defn}
\textbf{Note:} The Cauchy-Schwarz inequality implies that $|u_j ^* v_j| \leq \| u_j\| \|v_j\|$, so $\kappa(\lambda_j) \geq 1$, with equality when $u_j$ and $v_j$ are collinear. An eigenvalue for which $\kappa(\lambda_j)=1$ is called a normal eigenvalue; a matrix $A$ with all simple eigenvalues is normal if and only if $\kappa(\lambda_j)=1$ for all eigenvalues.

With this definition, we can find finer bounds for the pseudospectrum of a matrix; in particular, we can find bounds for the components of the pseduospectrum centered around each eigenvalue. The following theorem can be found for example in \cite{baumgartel1985analytic}.
\begin{thm}[Asymptotic pseudospectra inclusion regions]
Suppose $A \in \mathbb{C}^{N \times N}$ has $N$ distinct eigenvalues. Then, as $\epsilon \rightarrow 0$,
\[ \sigma_\epsilon (A) \subseteq \bigcup_{j=1}^{N} B\left(\lambda_j, \epsilon\kappa(\lambda_j)+\mathcal{O}\left(\epsilon^2\right)\right). \]
\label{thm: distinctBF}
\end{thm}

We can drop the $\mathcal{O}(\epsilon^2)$ term, for which we get an increase in the radius of our inclusion disks by a factor of $N$ \cite[Thm. 4]{bauer1960norms}.
\begin{thm}[Bauer-Fike theorem based on $\kappa(\lambda_j)$]
Suppose $A \in \mathbb{C}^{N \times N}$ has $N$ distinct eigenvalues. Then $\forall \epsilon > 0$,
\[ \sigma_\epsilon(A) \subseteq \bigcup_{j=1}^{N} B\left(\lambda_j, \epsilon N \kappa (\lambda_j)\right). \]
\end{thm}

The above two theorems give us upper bounds on the pseudospectra of $A$ only when $A$ has $N$ distinct eigenvalues. These results can be generalized for matrices that do not have distinct eigenvalues.  
The following is proven in \cite[\S 52]{trefethen2005spectra}.

\begin{thm}[Asymptotic formula for the resolvent norm]
Let $\lambda_j \in \sigma(A)$ be an eigenvalue of with $k_j$ the size of the largest Jordan block associated to $\lambda_j$. For any $z \in \sigma_\epsilon(A)$, for small enough $\epsilon$,
\[ |z-\lambda_j| \leq (C_j \epsilon)^\frac{1}{k_j}, \]
where $C_j= \| V_j T_j^{k_j -1} U_j^*\|$ and $T=J - \lambda I$.
\end{thm}

We extend these results by providing  lower bounds for arbitrary matrices, as well as explicit formulas for the $\epsilon$-pseudospectra of $2\times 2$ matrices.

\section{Pseudospectra of $2 \times 2$ Matrices} \label{sec: 2x2}

The following section presents a complete characterization of the $\epsilon$-pseudospectrum of any $2 \times 2$ matrix. We classify matrices by whether they are diagonalizable or non-diagonalizable and determine the $\epsilon$-pseudospectra for each class. We begin with an explicit formula for computing the norm of a $2 \times 2$ matrix.

Let $A =
\begin{pmatrix}
a & b \\
c & d
\end{pmatrix}$,
with $a, b, c, d \in \mathbb{C}$. Let $s_{\max}$ denote the largest singular value of $A$.

Then,
\begin{align}
||A||^2 	&= s_{\max} =\frac{\Tr(A^\ast A) + \sqrt{\Tr(A^\ast A)^2 - 4 \det(A^\ast A) }}{2} \label{matrixnorm}.
\end{align}

\subsection{Non-diagonalizable $2 \times 2$ Matrices} \label{subsec: non-diag}

Any $2 \times 2$ matrix that is non-diagonalizable must have exactly one eigenvalue of geometric multiplicity one. In this case, we can Jordan-decompose the matrix and use the first definition of pseudospectra to show that $\sigma_\epsilon(A)$ must be a perfect disk. 

\begin{prop} Let $A$ be any non-diagonalizable $2 \times 2$ matrix, and let $\lambda$ denote the eigenvalue of $A$. Write $A=VJV\inv$ where
\begin{equation} V=\begin{pmatrix} a&b\\c&d \end{pmatrix}, \qquad J=\begin{pmatrix} \lambda & 1 \\ 0& \lambda \end {pmatrix}. \label{eq: 22ndsetup} \end{equation}
Given any $\epsilon$,
\begin{equation} \sigma_\epsilon(A) = B\left(\lambda, |k| \right) \end{equation}
where
\begin{equation}
|k| = \sqrt{ C \epsilon + \epsilon^2 } \qquad \text{and} \qquad C=\frac{|a|^2+|c|^2}{|ad-bc|}.    \label{eq: 2x2nondiag}
\end{equation}
\label{nondiagonalizable2x2}
\end{prop}
\begin{proof}
Let $z= \lambda + k$ where $k \in \mathbb{C}$. Then we have $(z-A)\inv = V (z-J)\inv V\inv$.

Taking the norm, this yields
\[ \|(z-A)\inv \| = \frac{\|M\|}{|k^2(ad-bc)|}, \quad \text{where}\quad   M = \begin{pmatrix} adk - ac - bck & a^2 \\ -c^2 & -bck + ac +adk \end{pmatrix} \]

From \eqref{matrixnorm}, we obtain that
\[ \epsilon\inv < \|(z-A)\inv \| =\frac{\sqrt{\Tr(M^*M) + \sqrt{\Tr(M^*M)^2 - 4 \det(M^*M)}}}{|k|^2 |ad-bc|\sqrt{2}}. \]
Note that this function depends only on $|k|=|z-\lambda|$; thus for any $\epsilon$, $\sigma_\epsilon(A)$ will be a disk. Solving for $k$ to find the curve bounding the pseudospectrum, we obtain
\[ |k| = \sqrt{ \frac{|a|^2 + |c|^2}{|ad-bc|}\epsilon + \epsilon^2 }. \]\end{proof}

\subsection{Diagonalizable $2 \times 2$ Matrices} \label{subsec: diag}

Diagonalizable $2 \times 2$ matrices must have two distinct eigenvalues or be a multiple of the identity matrix. In either case, the pseudospectra can be described by the following proposition. 
\begin{prop} Let $A$ be any diagonalizable $2 \times 2$ matrix and let $\lambda_1, \lambda_2$ be the eigenvalues of A and $v_1, v_2$ be the eigenvectors associated with the eigenvalues. Then the boundary of $\sigma_\epsilon(A)$ is the set of points z that satisfy the equation
\begin{equation} (\epsilon^2 - |z-\lambda_1|^2)(\epsilon^2 - | z - \lambda_2|^2) - \epsilon^2 |\lambda_{1}-\lambda_{2}|^{2} \cot^2(\theta)=0, \end{equation}
where $\theta$ is the angle between the two eigenvectors.
\label{2x2eq}
\end{prop}
\label{diag22}
\begin{proof}
Since $A$ is diagonalizable, we can write $A=VDV^{-1}$ where
\begin{equation}V=\begin{pmatrix} a&b\\c&d \end{pmatrix} \qquad D=\begin{pmatrix} \lambda_1 &0\\ 0& \lambda_2 \end{pmatrix} \label{eq: 22dsetup} \end{equation}

Without loss of generality, let $z=\lambda_1+k$.

Let $\gamma=\lambda_1-\lambda_2$ and $r=ad-bc$. Then,
\[ \left \| (z-A)\inv \right \| = \left \| V (z-D)\inv V\inv \right\| = \frac{\|M\|}{| rk(\gamma +k)|},\] 
where \[M=\begin{pmatrix} ad\gamma + rk & -ab\gamma \\ cd\gamma & -bc\gamma + rk \end{pmatrix}.\]

Calculating $\Tr(M^*M)$, we obtain
\begin{equation} \Tr(M^* M) = |r|^2 \left( |\gamma + k|^2 + |k|^2 + |\gamma|^2 \cot^2 \theta \right), \label{simplifiedtrace} \end{equation}
where $\theta$ is the angle between the two eigenvectors, which are exactly the columns of $V$. 
For the determinant, we have
\begin{align}
\det(M^*M) 
&= |r|^4|k|^2|k+\gamma|^2
\end{align}
Plugging the above into equation \eqref{matrixnorm}, we get
\[\quad \epsilon\inv=\|(z-A)\inv \| = \sqrt{\frac{\Tr(M^*M)+ \sqrt {\Tr(M^*M)^2 - 4 \det(M^*M)}}{2|r|^2 | k(\gamma +k)|^2}}. \]

Re-writing and simplifying, we obtain the curve describing the boundary of the pseudospectrum:
\[(\epsilon^2 - |k|^2)(\epsilon^2 - |k+\gamma|^2) - \epsilon^2 |\gamma|^2 \cot^2\theta = 0. \]
\end{proof}

Note that for normal matrices, the eigenvectors are orthogonal. Therefore the equation above reduces to
\begin{equation} (\epsilon^2 - |k|^2)(\epsilon^2 - |k+\gamma|^2)  = 0  \end{equation}
which describes two disks of radius $\epsilon$ centered around $\lambda_1, \lambda_2$, as we expect.

When the matrix only has one eigenvalue and is still diagonalizable (i.e. when it is a multiple of the identity), then we obtain
\[ (\epsilon^2 - |k|^2)^2  = 0, \]
which is a disk of radius $\epsilon$ centered around the eigenvalue.

One consequence to note of Proposition \ref{2x2eq} is that the shape of $\sigma_\epsilon(A)$ is dependent on both the eigenvalues \emph{and} the eigenvectors of the matrix $A$.
Another less obvious consequence is that the pseudospectrum of a $2 \times 2$ matrix approaches a union of disks as $\epsilon$ tends to 0. 
\begin{prop}
Let $A$ be a diagonalizable $2 \times 2$ matrix with two distinct eigenvalues, $\lambda_1, \lambda_2$. Then, $\sigma_\epsilon(A)\restriction_{\lambda_i}$ asymptotically tends toward a disk.
In particular, \[\frac{r_\text{max}(\lambda_i)}{r_\text{min}(\lambda_i)} = 1 + \mathcal{O}(\epsilon),\] where $r_\text{max}(\lambda_i), r_\text{min}(\lambda_i)$ are the maximum and minimum distances from $\lambda_i$, to $\partial \sigma_\epsilon(A) \restriction_{\lambda_i}$.
Moreover, for $A$ diagonalizable but not normal, $\sigma_\epsilon(A)\restriction_{\lambda_i}$ is never a perfect disk.
\label{2x2disks}
\end{prop}

\begin{proof}

Let $\epsilon$ be small enough so that the $\epsilon$-pseudospectrum is disconnected. Without loss of generality, we will consider $\sigma_\epsilon(A)\restriction_{\lambda_1}$. 

Let $z_\text{max} \in \partial \sigma_\epsilon(A)$ such that $|z_\text{max}-\lambda_1|$ is a maximum. Set $r_\text{max}(\lambda)=|z_\text{max}-\lambda|$. Consider the line joining $\lambda_1$ and $\lambda_2$.
Suppose for contradiction that $z_\text{max}$ did not lie on this line. Then, rotate $z_\text{max}$ in the direction of $\lambda_2$ so that it is on this line, and call this new point $z'$.
Note that $|z'-\lambda_2| < |z_\text{max} - \lambda_2|$, but $|z'-\lambda_1| = |z_\text{max} - \lambda_1|$. As such, we get that
\[ (|z'-\lambda_1|^2-\epsilon^2)(| z' - \lambda_2|^2-\epsilon^2) < (|z_\text{max} - \lambda_1|^2-\epsilon^2)(|z_\text{max} - \lambda_2|^2-\epsilon^2) = \epsilon^2 |\lambda_1-\lambda_2|^2 \cot^2 \theta \]
Thus, from Proposition \ref{2x2eq} $z' \in \sigma_\epsilon(A)$ but $z'$ is not on the boundary of $\sigma_\epsilon(A)$. Starting from $z'$ and traversing the line joining $\lambda_1$ and $\lambda_2$, we can find $z'' \in \partial \sigma_\epsilon(A)$ such that $|z''-\lambda_1| > |z'-\lambda_1| = |z_\text{max} - \lambda_1|$. This contradicts our choice of $z_\text{max}$ and so $z_\text{max}$ must be on the line joining $\lambda_1$ and $\lambda_2$.
A similar argument shows that $z_\text{min}$ must also be on this line, where $z_\text{min} \in \partial \sigma_\epsilon(A)$ such that $r_\text{min}=|z_\text{min} - \lambda_1|$ is a minimum.

Since $z_\text{max}$ is on the line joining $\lambda_1$ and $\lambda_2$, we have the exact equality 
\[|z_\text{max} - \lambda_2| = |z_\text{max} - \lambda_1 | + |\lambda_2-\lambda_1 |. \]
Let $y = | \lambda_2 - \lambda_1 |$.
The equation describing $r_\text{max}(\lambda_1)$ becomes
\[ \left(r_\text{max}(\lambda_1)^2-\epsilon^2\right) \left( (y - r_\text{max} (\lambda_1))^2 - \epsilon^2\right) = \epsilon^2 y^2 \cot^2 \theta \]
Similarly, we can obtain the equation for $r_\text{min}(\lambda_1)$. Solving for $r_\text{max}(\lambda_1)$ and $r_\text{min}(\lambda_1)$, we get
\begin{align} r_\text{max}(\lambda_1) = \frac{1}{2} \left( y - \sqrt{y^2 + 4 \epsilon^2 - 4 y \epsilon \csc \theta } \right) \\
r_\text{min}(\lambda_1) = \frac{1}{2} \left( \sqrt{y^2 + 4 \epsilon^2 + 4 y \epsilon \csc \theta } -y \right) \end{align}

For $\epsilon$ small, we can use the approximation $(1+\epsilon)^p = 1 + p \epsilon +  \mathcal{O}(\epsilon^2)$. Then,
\begin{align}
\frac{r_\text{max}(\lambda_1)}{r_\text{min}(\lambda_1)} &= \frac{1 - \sqrt{1 + 4 (\epsilon/y)^2 - 4 (\epsilon/y) \csc \theta } }{\sqrt{1 + 4 (\epsilon/y)^2 + 4 (\epsilon/y) \csc \theta} -1} \nonumber \\
&= \frac{ 1 + \eta \epsilon + \mathcal{O} (\epsilon^2) }{1 - \eta \epsilon + \mathcal{O} (\epsilon^2)}
\end{align}
where $\eta = \frac{\cos \theta \cot \theta}{y} $.
Using the geometric series approximation $\frac{1}{1-x} = 1+x+\mathcal{O}(x^2)$, we find that
\begin{align}
\frac{r_\text{max}}{r_\text{min}} 
&= 1 + \frac{(2 \cos \theta \cot \theta ) \epsilon}{y} + \mathcal{O}(\epsilon^2) \label{ratio}
\end{align}
Thus, $\sigma_\epsilon(A)$ tends towards a disk. Moreover, if $A$ is diagonalizable but not normal, then the eigenvectors are linearly independent but not orthogonal, so $\theta$ is not a multiple of $\pi/2$ or $\pi$, and therefore $\cos \theta \cot \theta \neq 0$ and $\frac{r_\text{max}(\lambda)}{r_\text{min}(\lambda)} \neq 1$.
\end{proof}

This result can be observed by looking at plots of the pseudospectra of diagonalizable $2 \times 2$ matrices.
\begin{figure}[h]
\includegraphics[scale=0.5]{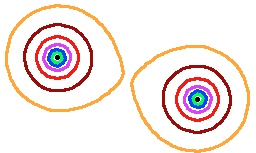} \qquad  \qquad \quad
\includegraphics[scale=.5]{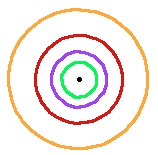}
\caption{$\epsilon$-pseudospectra of a diagonalizable $2 \times 2$ matrix}
\label{22andzoom}
\end{figure}

The image on the left shows the pseudospectra of a particular $2 \times 2$ matrix. One can see that for large enough values of $\epsilon$, the pseudospectra around either eigenvalue is not a perfect disk. The image on the right is the pseudospectra of the same matrix (restricted to one eigenvalue), with smaller values of epsilon. Here, the pseudospectra appear to converge to disks. We find that this result holds in general for any $N \times N$ matrix and this is proven in the following section. 

\section{Asymptotic Union of Disks Theorem} \label{sec: audit}
In Propositions \ref{nondiagonalizable2x2} and \ref{diag22}, we showed that the $\epsilon$-pseudospectra for all $2 \times 2$ matrices are disks or asymptotically converge to a union of disks. We now explore whether this behavior holds in the general case. It is possible to find matrices whose $\epsilon-$pseudospectra exhibit pathological properties for large $\epsilon$; for example, the non-diagonalizable matrix given in Figure \ref{fig: toeplitz} has, for larger $\epsilon$, an $\epsilon$-pseudospectrum that is not convex and not simply connected.

\begin{figure}
\begin{subfigure}[h!]{0.9\textwidth}
$ \begin{pmatrix}
-1 & -10 & -100 & -1000 & -10000 \\
0 & -1 & -10 & -100 & -1000 \\
0 & 0 & -1 & -10 & -100 \\
0 & 0 & 0 & -1 & -10 \\
0 & 0 & 0 & 0 & -1 
\end{pmatrix} $
\end{subfigure}
\hspace{-5cm} 
\begin{subfigure}[h!]{0.3\textwidth}
\includegraphics[scale=0.5]{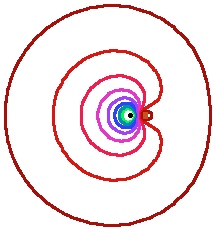}
\label{fig: toep}
\end{subfigure}
\caption{Pseudospectra of a Toeplitz matrix}
\label{fig: toeplitz}
\end{figure}

Thus, pseudospectra may behave poorly for large enough $\epsilon$; however, in the limit as $\epsilon \to 0$, these properties disappear and the pseudospectra behave as disks centered around the eigenvalues with well-understood radii. In order to understand this asymptotic behavior, we will use the following set-up (which follows \cite{moro1997lidskii}).

Let $A \in \mathbb{C}^{N \times N}$ and fix $\lambda \in \sigma(A)$. Write the Jordan decomposition of $A$ as such
\[ \begin{pmatrix}J & \\[0.5em] & \hat{J} \end{pmatrix} = \begin{pmatrix} Q \\[0.5em] \hat{Q}\end{pmatrix} A \begin{pmatrix} P & \hat{P} \end{pmatrix}, \qquad \begin{pmatrix} Q \\[0.5em] \hat{Q}\end{pmatrix}  \begin{pmatrix} P & \hat{P} \end{pmatrix} = I \]
where $J$ consists of Jordan blocks $J_1,\ldots,J_m$ corresponding to the eigenvalue $\lambda$. $\hat{J}$ consists of Jordan blocks corresponding to the other eigenvalues of $A$.

Let $n$ be the size of the largest Jordan block corresponding to $\lambda$, and suppose there are $\ell$ Jordan blocks corresponding to $\lambda$ of size $n \times n$.
Arrange the Jordan blocks in $J$ in weakly decreasing order, according to size. That is,
\[ \dim{(J_1)} = \cdots = \dim (J_\ell) > \dim{(J_{\ell+1})} \geq \cdots \geq \dim{(J_m)} \]
where $J_1, \ldots , J_\ell $ are $n \times n$.

Further partition $P$,
\[ P =\begin{pmatrix} P_1 \ , & \ldots \ , & P_\ell \ , & \ldots \ ,& P_m \end{pmatrix} \]
in a way that agrees with the above partition of $J$, so that the first column, $x_{j}$, of each $P_j$ is a right eigenvector of $A$ associated with $\lambda$. We also partition $Q$ likewise
 \[ Q = \begin{pmatrix}
 Q_1 \\[0.5em]
 \vdots \\
 Q_\ell \\[0.5em]
 \vdots \\[0.5em]
 Q_m \end{pmatrix}. \]
 The last row, $y_j$, of each $Q_j$ is a left eigenvector of $A$ corresponding to $\lambda$.
 \newline
 We now build the matrices
 \[ Y= \begin{pmatrix}
 y_1 \\[0.5em]
 y_2 \\[0.5em]
 \vdots\\[0.5em]
 y_\ell \end{pmatrix}, \qquad
 X=\begin{pmatrix} x_1 \ , & x_2\ , & \ldots \ ,& x_\ell \end{pmatrix}, \]
 where $X$ and $Y$ are the matrices of right and left eigenvectors, respectively, corresponding to the Jordan blocks of maximal size for $\lambda$.

The following theorem is presented by Moro, Burke, and Overton \cite{moro1997lidskii} and due to Lidskii \cite{Lidskii196673}.
\begin{thm}[Lidskii \cite{Lidskii196673}] \label{thm: Lidskii}
Given $l, n$ as defined above corresponding to the matrix $A$, there are $\ell n$ eigenvalues of the perturbed matrix $A+\epsilon E$ admitting a first order expansion
\[ \lambda_{j, k}(\epsilon) = \lambda + (\gamma_j \epsilon)^{1/n} + o(\epsilon^{1/n}) \]
for $j=1, \ldots, \ell,  \; \; k = 1, \ldots, n$, where $\gamma_j$ are the eigenvalues of $YEX$ and the different values of $\lambda_{j, k}(\epsilon)$ for $k =1, \ldots, n$ are defined by taking the distinct $n^{th}$ roots of $\gamma_j$.
\end{thm}

Lidskii's result can be interpreted in terms of the $\epsilon$-pseudospectrum of a matrix $A$ in order to understand the radii of $\sigma_\epsilon(A)$ as $\epsilon \to 0$.
\begin{thm}

Let $A \in \mathbb{C}^{N \times N}$. Let $\epsilon > 0$. Given $\lambda \in \sigma(A)$, for $\epsilon$ small enough, there exists a connected component $U \subseteq \sigma_\epsilon (A)$ such that $U \cap \sigma(A) = \lambda$; denote this component of the $\epsilon$-pseudospectrum  $\sigma_\epsilon(A) \restriction_\lambda $.

Then, as $\epsilon \to 0$,
\[ B(\lambda, (C\epsilon)^{1/n} + o(\epsilon^{1/n})) \subseteq \sigma_\epsilon(A) \restriction_\lambda \subseteq B(\lambda, (C\epsilon)^{1/n} + o(\epsilon^{1/n})) \]
where $C = \| XY \| $, with $X,Y$ defined above, and $n$ is the size of the largest Jordan block corresponding to $\lambda$.

\label{thm: AUDiT}
\end{thm}

\begin{proof}

\textbf{Lower Bound:} 
Give $E \in \mathbb{C}^{N \times N}$, let $\gamma_{\max}(E)$ be the largest eigenvalue of $YEX$. It is shown \cite[Theorem 4.2]{moro1997lidskii} that
\[ \alpha := \max_{\| E \| \leq 1} \gamma_{\max}(E) = \| XY \|, \]
Moreover, the $E$ that maximizes $\gamma$ is given by $E = vu$ where $v$ and $u$ are the right and left singular vectors of the largest singular value of $XY$, normalized so $\| v \| = \| u \| = 1$.
We claim that $B(\lambda, (|\alpha| \epsilon)^{1/n} + o(\epsilon^{1/n}) ) \subseteq \sigma_\epsilon(A) \restriction_\lambda$.

Fix $E = vu$, with $v, u$ defined above, fix $\theta \in [0, 2 n \pi]$, and define $\tilde{E} = e^{i \theta} E$. 
Note that $\gamma$ is an eigenvalue of $Y E X$ iff $e^{i \theta} \gamma$ is an eigenvalue of $Y\tilde{E}X$.
Since $\alpha$ is an eigenvalue of $E$, then $e^{i \theta} \alpha$ is an eigenvalue of $Y\tilde{E}X$.
Considering the perturbed matrix $A + \epsilon \tilde{E}$, theorem \ref{thm: Lidskii} implies that there is a perturbed eigenvalue $\lambda(\epsilon)$ of the form
\[ \lambda(\epsilon) = \lambda + (e^{i \theta} \alpha \epsilon)^{1/n} + o(\epsilon^{1/n}) \]
and thus $\lambda(\epsilon) \in \overline{\sigma_\epsilon(A) \restriction_\lambda} $.
Ranging $\theta$ from $0$ to $2 n \pi$, we get the desired result.

\textbf{Upper Bound:} 
Using the proof of \cite[Theorem 52.3]{trefethen2005spectra}, we know that asymptotically
\[ \sigma_\epsilon(A) \restriction_\lambda \subseteq B(\lambda, (\beta \epsilon)^{1/n} + o(\epsilon^{1/n})), \]
where $\beta = \| PD^{n-1} Q \| $ and $J= \lambda I +D$.
We claim $\beta = \| X Y \| = \alpha$.

Note that $D^{n-1} = \text{diag}[\Gamma_1, \ldots \Gamma_\ell, 0]$ where $\Gamma_k$ is a $n \times n$ matrix with a 1 in the top right entry and zeros elsewhere.
We find
\setcounter{MaxMatrixCols}{20}
\[ PD^{n-1} =
\begin{pmatrix}
&	&	&	&	&	& \\
&	&	&	&	&	& \\
&	& x_1	&	& x_2	& 	& \cdots & 	&  x_\ell & & &\\
&	&	&	&	&	& \\
&	&	&	&	&	&
\end{pmatrix}.
\]
This then gives
\[ PD^{n-1} Q =
\begin{pmatrix}
XY & 0 \\
0 & 0
\end{pmatrix}.
\]
Thus $\beta = \| PD^{n-1} Q \|  = \| XY \| = \alpha$.
\end{proof}

We present special cases of matrices to explore the consequences of Theorem \ref{thm: AUDiT}.

{\bf Special Cases:}

\begin{enumerate}
\item $\lambda$ is simple. 

Then, $n=1$ and $X$ and $Y$ become the right and left eigenvectors $x$ and $y^\ast$ for $\lambda$, respectively.
Hence, $C = \| XY \| = \| x y^\ast \| = \|x\| \|y\| = \kappa(\lambda)$ where we normalize so that $| y^\ast x | = 1$.
Then, Theorem~\ref{thm: AUDiT} becomes
\[ \sigma_\epsilon(A) \restriction_\lambda \approx B(\lambda, \kappa(\lambda) \epsilon ) \]
which matches with Theorem~\ref{thm: distinctBF}.

\item $\lambda$ has geometric multiplicity 1.

In this case we obtain the same result for when $\lambda$ is simple, except $n$ may not equal 1. In other words,
\[ \sigma_\epsilon(A) \restriction_\lambda \approx B(\lambda, (\kappa(\lambda) \epsilon)^{1/n} ). \]

\item $A \in \mathbb{C}^{2\times2}$.

There are two cases, as in Section \ref{sec: 2x2}:
\newline
First, assume $A$ is non-diagonalizable. In this case, $A$ only has one eigenvalue, $\lambda$. Writing $A=VJV\inv$, where $V$ and $J$ are as defined in equation \eqref{eq: 22ndsetup}, we have that,
\[X=\begin{pmatrix} a&c \end{pmatrix}^T,  \qquad Y= \frac{1}{ad-bc}\begin{pmatrix} -c &  a \end{pmatrix}.\]
From Theorem~\ref{thm: AUDiT}, we then have that as $\epsilon \to 0$,
\[ \sigma_\epsilon(A) \approx B\left( \lambda, \left(\frac{|a|^2+|c|^2}{|ad-bc|}\epsilon \right)^{1/2} + o\left(\epsilon^{1/2}\right)\right). \]
This agrees asymptotically with equation~\ref{eq: 2x2nondiag}; however \ref{eq: 2x2nondiag} gives an explicit formula for $\sigma_\epsilon(A)$.

In the case where $A$ is diagonalizable, $A$ has two eigenvalues, $\lambda_1$ and $\lambda_2$. Again, we write $A=VDV\inv$ where $V$ and $D$ are as defined in equation \eqref{eq: 22dsetup}. From this, we have
\[ \|XY\|=\frac{\left(|a|^2+|c|^2\right)\left(|b|^2+|d|^2\right)}{|ad-bc|}=\csc \theta. \]

Thus, as $\epsilon \to 0$, we have from Theorem~\ref{thm: AUDiT}:
\[B\left( \lambda, \left(\csc \theta \right) \epsilon  + o\left(\epsilon\right)\right) \subseteq \sigma_\epsilon(A) \subseteq B\left( \lambda, \left(\csc \theta \right) \epsilon  + o\left(\epsilon\right)\right). \]

So,
\[ \frac{r_\text{max}}{r_\text{min}} = \frac{\left(\csc \theta \right) \epsilon  + o\left(\epsilon\right)}{\left(\csc \theta \right) \epsilon  + o\left(\epsilon\right)} = 1 + o(1) \]

This agrees with the ratio we obtain from the explicit formula for diagonalizable $2\times 2$ matrices; however, equation \eqref{ratio} gives us more information on the $o(1)$ term.

\item $A$ is a Jordan block.

From \cite[pg. 470]{trefethen2005spectra}, we know that the $\epsilon$-pseudospectrum of the Jordan block is exactly a disk about the eigenvalue of $J$ of some radius. An explicit formula for the radius remains unknown, however we can use Theorem \ref{thm: AUDiT} to find the asymptotic behavior.

\begin{prop}[Asymptotic Bound]  Let $J$ be an $N \times N$ Jordan block. Then \[\sigma_\epsilon(J) = B(\lambda, \epsilon^{1/N}+ o(\epsilon^{1/N})). \label{JordanAsymptotic} \]

\begin{proof} The $N \times N$ Jordan block has left and right eigenvectors $u_j$ and $v_j$ where $\|u_j\|=1$ and $\|v_j\|=1$. So, from Theorem \ref{thm: AUDiT}, we find $C=\|XY\|=\|v_ju_j\|=1$. Thus, \[\sigma_\epsilon(J) \approx B(\lambda, \epsilon^{1/N}+o(\epsilon^{1/N})). \]

\end{proof}
\end{prop}

By a simple computation, we can also get a better explicit lower bound on the $\epsilon$-pseudospectra of an $N \times N$ Jordan block, that agrees with our asymptotic bound.

\begin{prop} Let J be an $ N \times N$ Jordan block. 
 Then,
\[ B \left(\lambda, \sqrt[N]{\epsilon(1+\epsilon)^{N-1}}\right) \subseteq \sigma_\epsilon(J). \]
\label{jordanblock}
\end{prop}

\begin{proof}
We use the second definition for $\sigma_\epsilon(J)$. Let
\[ E=\begin{pmatrix} 
 0& k &  &  &  \\ 
&  0& k &  &  \\
 &  & \ddots & \ddots &  \\
 &  &  &  \ddots& k \\
k & &  &  &  0\end{pmatrix} \]
where $|k| < \epsilon$, and note that $\|E\|<\epsilon$. 
We take $\det(J+E-z I)$ and set it equal to zero to find the eigenvalues of $J+E$. 
\begin{align*} 0 &= \det(J+E-z I) \\
&= \det \begin{pmatrix} 
\lambda-z & k+1 &  &  & \\ 
& \lambda-z & k+1 &  &  \\
&  & \ddots &  \ddots&  \\
&  &  & \ddots & k+1 \\
k &  &  &  & \lambda-z
 \end{pmatrix} \\
 &= (\lambda-z)^N + (-1)^{N-1}k(1+k)^{N-1} \\
 &= (-1)^{N-1}((z-\lambda)^N+k(1+k)^{N-1}); \\
\iff (z-\lambda)^N &= k(1+k)^{N-1}\\
 z-\lambda &=\sqrt[N]{k(1+k)^{N-1}}. 
 \end{align*}
 
 So, $B \left(\lambda, \sqrt[N]{\epsilon(1+\epsilon)^{N-1}}\right) \subseteq \sigma_\epsilon(J)$.
 \end{proof}

\end{enumerate}

\section{Pseudospectra of bidiagonal matrices} \label{sec: bidiag}

In this section we consider bidiagonal matrices, a class of matrices with important applications in spectral theory and mathematical physics. We investigate the pseudospectra of periodic bidiagonal matrices and show that the powers $n$ and the coefficients $C$ in Theorem \ref{thm: AUDiT} can be computed explicitly. 
We consider the coefficients $\{ a_k \}_{k = 1}^{N}$ and $\{ b_k \}_{k = 1}^{N-1}$ which define the bidiagonal matrix 
\[
A = \textrm{bidiag} \left(\{ a_k \}_{k = 1}^{N}, \{ b_k \}_{k = 1}^{N-1}\right) = \left(
                                                                         \begin{array}{cccccc}
                                                                           a_1 & b_1 &  &  &  &  \\
                                                                            & a_2 & b_2 &  &  &  \\
                                                                            &  & \ddots & \ddots &  &  \\
                                                                            &  &  & \ddots & \ddots &  \\
                                                                            &  &  &  & a_{N-1} & b_{N-1} \\
                                                                            &  &  &  &  & a_N \\
                                                                         \end{array}
                                                                       \right)
\]

Note that if $b_i = 0$ for some $i$, then the matrix $A$ ``decouples" into the direct sum $A = \textrm{bidiag} \left(\{ a_k \}_{k = 1}^{i}, \{ b_k \}_{k = 1}^{i-1}\right) \oplus \textrm{bidiag} \left(\{ a_k \}_{k = i+1}^{N}, \{ b_k \}_{k = i+1}^{N-1}\right)$ and by Theorem \ref{directsum} the pseudospectrum of $A$ is the union of pseudospectra of smaller bidiagonal matrices. Therefore we can assume, without loss of generality, that $b_i \neq 0$ for any $i \in \{ 1, 2, \ldots, N - 1\}$.

Note also that the eigenvalues of $A$ are $\{ a_1, a_2, \ldots, a_n \}$ and some eigenvalues may be repeated in the list. In order to apply Theorem \ref{thm: AUDiT} we have to  find the dimension of the largest Jordan block associated to each eigenvalue of the matrix $A$. The following proposition addresses this question:

\begin{prop}
Let $A = \mathrm{bidiag} \left(\{ a_k \}_{k = 1}^{N}, \{ b_k \}_{k = 1}^{N-1}\right)$ with $b_i \neq 0$ for any $i$ and suppose that $a$ is an eigenvalue of $A$. Then $\dim N(A - a I) = 1$, where $N(A - a I)$ is the eigenspace corresponding to the eigenvalue $a$ of the matrix $A$.
\label{geometricmultiplicity}
\end{prop}

\begin{proof} 

Suppose $a = a_{i_1} = a_{i_2} = \cdots = a_{i_m}$ where $1 \leq i_1 < i_2 < \cdots < i_m \leq N$ and $a \neq a_k$ for any $k \in \{ 1, 2, \ldots, n\} \setminus \{i_1, i_2, \ldots, i_m\}$.
We have
\[
A - a I = \left(
                                                                                              \begin{array}{cccccccc}
                                                                                                a_1 - a & b_1 &  &  &  &  &  &  \\
                                                                                                 & a_2 - a & \ddots &  &  &  &  &  \\
                                                                                                 &  & \ddots & b_{i_1 - 2} &  &  &  &  \\
                                                                                                 &  &  & a_{i_1 - 1} - a & b_{i_1 - 1} &  &  &  \\
                                                                                                 &  &  &  & 0 & b_{i_1} &  &  \\
                                                                                                 &  &  &  &  & a_{i_1 + 1} - a & \ddots &  \\
                                                                                                 &  &  &  &  &  & \ddots & b_{N-1} \\
                                                                                                 &  &  &  &  &  &  & a_N - a \\
                                                                                              \end{array}
                                                                                            \right)
\]
Let us denote by $\{ \textbf{c}_1, \textbf{c}_2, \ldots, \textbf{c}_N \}$ the columns of $A - a I$ and by $\{ \textbf{e}_1, \textbf{e}_2, \ldots, \textbf{e}_N \}$ the standard canonical basis in $\R^N$. Since $b_k \neq 0$ for any $k \in \{ 1, 2, \ldots, N - 1 \}$ we obtain that columns $\{ \textbf{c}_2, \textbf{c}_3, \ldots, \textbf{c}_{N} \}$ are linearly independent. Moreover, we also have $\textrm{Span} ( \textbf{c}_2, \textbf{c}_3, \ldots, \textbf{c}_{i_1} ) = \textrm{Span} ( \textbf{e}_1, \textbf{e}_2, \ldots, \textbf{e}_{i_1 - 1} )$, which in turn implies that $\textbf{c}_1 \in \textrm{Span} \{ \textbf{c}_2, \textbf{c}_3, \ldots, \textbf{c}_{i_1} \}$.
We conclude that the rank of the matrix $A - a I$ is $N - 1$, hence $\dim N (A - a I) = 1$. 
\end{proof}

The previous proposition implies that, under the assumption $b_i \neq 0$ for any $i$, if $a$ is an eigenvalue of the matrix $A = \textrm{bidiag} \left(\{ a_k \}_{k = 1}^{N}, \{ b_k \}_{k = 1}^{N-1}\right)$ of algebraic multiplicity $m$, then there is only one Jordan block associated to the eigenvalue $a$.

We now consider the special case of periodic bidiagonal matrices.
Let $A$ be an $N \times N$ matrix with period $k$ on the main diagonal and nonzero superdiagonal entries
\[ A =
\begin{pmatrix}
a_1 	& b_1 	&  		&  		&  		&  		&  		&    		\\
 	& \ddots 	& \ddots	& 		&  		&  		&  		&    		\\
	& 		&  a_k	& \ddots	&  		&  		&  		&    		\\
	& 		& 		& \ddots	&  \ddots	&  		&  		&    		\\
 	& 		& 		& 		&  \ddots	&  \ddots	&  		&    		\\
 	&  		&  		&  		& 	 	&  a_1	& \ddots	&  		\\
 	&  		&  		&  		& 	 	&  		& \ddots	& b_{N-1}	\\
 	&  		&  		&  		& 	 	&  		& 		& a_r 	\\	
\end{pmatrix}
\]
We have from Theorem~\ref{thm: AUDiT} and Proposition \ref{geometricmultiplicity} that
\[\sigma_\epsilon(A) \approx \bigcup_{j=1}^{k} B\left(a_j, (C_j \epsilon)^{\frac{1}{n_j}}\right),\]
where $n_j$ is the size of the Jordan block corresponding to $a_j$ and also the number of times $a_j$ appears on the main diagonal. Moreover, the constant $C_j$ that multiplies any eigenvalue $a_j$ is simply $C_j=\|v_j\| \|u_j\|$, where $v_j$ and $u_j$ denote the right and left eigenvectors, respectively. We will give the explicit expressions for $v_j$ and $u_j$. 

We will begin by introducing $\epsilon$-pseudospectrum for simple special cases which lead to the most general case.

The cases will be presented as follows:

\begin{enumerate}
\item Let $A$ be a $kn \times kn$ matrix with $a_1, \ldots, a_k$ distinct.
\item Let $A$ be an $N \times N$ matrix with $a_1, \ldots, a_k$ distinct.
\item \textbf{General Case:} Let $A$ be an $N \times N$ matrix with $a_1, \ldots, a_k$ not distinct.

\end{enumerate}

To shorten notation for the rest of this section, we define
\[f(x)=
\begin{cases}
x , \quad x\neq 0\\
1  , \quad x=0
\end{cases}.
\]

\textbf{Case 1:}
\begin{itemize}
\item
The size of $A$ is $kn \times kn$.
\item
The $a_i$'s are distinct.
\end{itemize}

We write the elements of the superdiagonal as $b_1, b_2, \ldots , b_{N-1}$. Let $p=k(n-1)+j$. 

We have that:
\[ v_j =
\begin{pmatrix}
\frac{b_1 \cdots b_{j-1}}{f(a_j-a_1) \cdots f(a_j-a_{j-1})} \\[0.5em]
\frac{b_2 \cdots b_{j-1}}{f(a_j-a_2) \cdots f(a_j-a_{j-1})} \\[0.5em]
\vdots \\[0.5em]
\frac{b_{j-1}}{f(a_j-a_{j-1})} \\[0.5em]
1 \\
0 \\
\vdots \\
0
\end{pmatrix}, \qquad u_j^\ast = \left(\frac{1}{(f(a_1-a_j)\cdots f(a_k-a_j))^{n-1}}\right)
\begin{pmatrix}
0\\ \vdots \\ 0 \\ 1 \\[0.5em] \frac{b_j\cdots b_{p} }{f(a_{j+1}-a_j) }  \\[0.5em]  \vdots \\[0.5em]  \frac{b_{j} \cdots b_{N-2}}{f(a_{j+1}-a_j) \cdots f(a_{k-1}-a_j)}  \\[0.5em]  \frac{b_{j} \cdots b_{N-1}}{f(a_{j+1}-a_j) \cdots f(a_k-a_j)} 
\end{pmatrix}^T.
\]

Direct computation will show that these are indeed left and right eigenvectors associated with any eigenvalue $a_j$.

\textbf{Case 2:}

\begin{itemize}
\item
The size of $A$ is $N \times N$.
\item
The $a_i$'s are distinct.
\end{itemize}
We relax our assumption that the size of our matrix is $kn \times kn$, for period $k$ on the diagonal.
Let $n, r$ be such that $N = kn+r$, where $0 < r \leq k$. In other words, $a_r$ is the last entry on the main diagonal, so the period does not necessarily complete.

For $a_j$, the right eigenvector is given by
\[ v_j =
\begin{pmatrix}
\frac{b_1 \cdots b_{j-1}}{f(a_j-a_1) \cdots f(a_j-a_{j-1})}, &
\frac{b_2 \cdots b_{j-1}}{f(a_j-a_2) \cdots f(a_j-a_{j-1})}, & 
\cdots, &
\frac{b_{j-1}}{f(a_j-a_{j-1})}, &
1, &
0, &
\cdots, &
0
\end{pmatrix}^T.
\]

We split up the formula for the left eigenvectors into two cases:
\begin{enumerate}
\item
 $1 \leq j \leq r$
 \item
 $r < j \leq k$.
 \end{enumerate}

(1)\qquad \emph{ $1 \leq j \leq r$}
\newline
On the main diagonal, there are $n$ \emph{complete} blocks with entries $a_1, \ldots, a_k$, and one \emph{partial} block at the end with entries $a_1, \ldots a_r$.
In the first case, when $1 \leq j \leq r$, then $a_j$ is in this last partial block.
In this case then, let $p=kn+j$.

We have that
\[
u_j = \mu_j
\begin{pmatrix}
0 \\
\vdots \\
0 \\
(b_j\cdots b_{p-1})\cdot f(a_{j+1}-a_j) \cdots f(a_r-a_j) \\[0.5em]
(b_j\cdots b_{p} )\cdot f(a_{j+2}-a_j) \cdots f(a_r-a_j) \\[0.5em]
\vdots \\[0.5em]
(b_j\cdots  b_{N-2}) \cdot f(a_r-a_j) \\
b_j\cdots  b_{N-1}
\end{pmatrix}^T,
\]

where \[\mu_j = \frac{f(a_1-a_j)\cdots f(a_{j-1}-a_j)}{[f(a_1-a_j) \cdots f(a_k-a_j)]^n f(a_1-a_j) \cdots f(a_r-a_j)}.\]
\newline
\newline
(2) \qquad \emph{$r < j \leq k$}

In this case, $a_j$ is in the last \emph{complete} block. Now, let $p=k(n-1)+j$.

We have that
\[
u_j = \mu_j
\begin{pmatrix}
0 \\
\vdots \\
0 \\
(b_j \cdots b_{p-1})\cdot f(a_1-a_j) \cdots f(a_r-a_j) f(a_{j+1}-a_j) \cdots f(a_k-a_j) \\[0.5em]
(b_j \cdots b_{p}) \cdot f(a_1-a_j) \cdots f(a_r-a_j) f(a_{j+2}-a_j) \cdots f(a_k-a_j) \\[0.5em]
\vdots \\
(b_j \cdots b_{p+k-j-1})\cdot f(a_1-a_j) \cdots f(a_r-a_j)  \\[0.5em]
(b_j \cdots b_{p+k-j}) \cdot  f(a_2-a_j) \cdots f(a_r-a_j) \\[0.5em]
\vdots \\
(b_j \cdots b_{N-2})\cdot f(a_r-a_j) \\[0.5em]
b_j \cdots b_{N-1}
\end{pmatrix},
\]
again where \[\mu_j = \frac{f(a_1-a_j)\cdots f(a_{j-1}-a_j)}{[f(a_1-a_j)\cdots  f(a_k-a_j)]^n f(a_1-a_j) \cdots f(a_r-a_j)}. \]

\textbf{Case 3: General Case. }

\begin{itemize}
\item
The size of $A$ is $N \times N$.
\item
The $a_i$'s are not distinct for $1\leq i \leq k$
\end{itemize}

Let $A$ be a $N \times N$ periodic bidiagonal matrix with period $k$ on the main diagonal. Let  $n, r$ be such that $N=kn+r$, where $0 < r \leq k$. Write $a_1, \ldots, a_k$ for the entries on the main diagonal ($a_i$'s not distinct) and $b_1, \ldots, b_{N-1}$ for the entries on the superdiagonal. Let $a_r$ be the last entry on the main diagonal.

We can explicitly find the left and right eigenvectors for any eigenvalue, $\alpha$. Suppose $\alpha$ first appears in position $\ell$ of the period $k$. Then the corresponding right eigenvector for $\alpha$ is the same form as $v_\ell$ in case 2. That is,
\[ v_\ell =
\begin{pmatrix}
\frac{b_1 \cdots b_{\ell-1}}{f(a_\ell-a_1) \cdots f(a_\ell-a_{\ell-1})} , &
\frac{b_2 \cdots b_{\ell-1}}{f(a_\ell-a_2) \cdots f(a_\ell-a_{\ell-1})} , &
\cdots , & 
\frac{b_{\ell-1}}{f(a_\ell-a_{\ell-1})} , &
1 , &
0 ,&
\cdots ,&
0
\end{pmatrix}^T.
\]
The corresponding left eigenvector for $\alpha$ depends on the first and last positions of $\alpha$. Let $k(n-1)=\ell q+s$ and set $q \equiv m\pmod{k}$. We split up the formula for the left eigenvector of $\alpha$ into two cases, which again mirror the formulas given in case 2:
\begin{enumerate}
\item
 $1 \leq \ell \leq r$
 \item
 $r < \ell \leq k$.
 \end{enumerate}
For both of these two cases, we define
\[g(b_{i})=
\begin{cases}
b_{i}, &\mbox i \geq p\\
1, &\mbox i < p
\end{cases}.
\]
(1)\qquad \emph{ $1 \leq \ell \leq r$}
\newline
In this case then, $\alpha$ appears in the partial block. Let $p=kn+\ell$. We have that
\[
u_\ell = \mu_\ell
\begin{pmatrix}
0 \\
\vdots \\
0 \\
g(b_{p+m-\ell-1})f(a_{m+1}-a_\ell) \cdots f(a_r-a_\ell) \\[0.5em]
g(b_{p+m-\ell-1})g(b_{p+m-l}) f(a_{m+2}-a_\ell) \cdots f(a_r-a_\ell) \\[0.5em]
\vdots \\[0.5em]
g(b_{p+m-\ell-1}) \cdots g(b_{N-2})  f(a_r-a_\ell) \\
g(b_{p+m-\ell-1})\cdots g(b_{N-1})
\end{pmatrix},
\]

where $\mu_\ell = \frac{b_\ell \cdots b_{p-1} f(a_1-a_\ell)\cdots f(a_{\ell-1}-a_\ell)}{[f(a_1-a_\ell)\cdots f(a_k-a_\ell)]^n f(a_1-a_\ell) \cdots f(a_r-a_\ell)}$.

(2) \qquad \emph{$r < \ell \leq k$}
\newline
In this case, $\alpha$ is in the last \emph{complete} block. Here, we let $p=k(n-1)+\ell$. Now, we have
\[
u_\ell = \mu_\ell
\begin{pmatrix}
0 \\
\vdots \\
0 \\
g(b_{p+m-\ell-1})f(a_1-a_\ell) \cdots (a_r-a_\ell) f(a_{m+1}-a_\ell) \cdots f(a_k-a_\ell) \\[0.5em]
g(b_{p+m-\ell-1}) g(b_{p+m-l} )f(a_1-a_\ell) \cdots f(a_r-a_\ell) f(a_{m+2}-a_\ell) \cdots f(a_k-a_\ell) \\[0.5em]
\vdots \\
g(b_{p+m-\ell-1})\cdots g(b_{p+k-\ell-1}) f(a_1-a_\ell) \cdots f(a_r-a_\ell)  \\[0.5em]
g(b_{p+m-\ell-1})\cdots g(b_{p+k-\ell}) f(a_2-a_\ell) \cdots f(a_r-a_\ell) \\[0.5em]
\vdots \\
g(b_{p+m-\ell-1}) \cdots g(b_{N-2}) f(a_r-a_\ell) \\[0.5em]
g(b_{p+m-\ell-1})\cdots g(b_{N-1})
\end{pmatrix},
\]
where \[\mu_\ell = \frac{\left(b_\ell \cdots b_{p-1}\right)\cdot f(a_1-a_\ell)\cdots f(a_{\ell-1}-a_\ell)}{[f(a_1-a_\ell)\cdots f(a_k-a_\ell)]^n f(a_1-a_\ell) \cdots f(a_r-a_\ell)}. \]

From these formulas, we can find the eigenvectors, and hence the asymptotic behavior of the $\epsilon$-pseudospectrum for any bidiagonal matrix,
\[ \sigma_\epsilon(A) \approx \bigcup_{j=1}^{k} B\left(a_j, (C_j\epsilon)^{\frac{1}{n_j}}\right) \]
where $C_j=\|v_j\|\|u_j\|$ and $n_j$ is the size of the Jordan block corresponding to $a_j$.

\textbf{Note:} Let $A$ be a periodic, bidiagonal matrix and suppose $b_i=0$ for some $i$. Then the matrix decouples into the direct sum of smaller matrices, call them $A_1, \ldots , A_n$. To find the $\epsilon$-pseudospectrum of $A$, apply the same analysis to these smaller matrices, and from Theorem \ref{directsum}, we have that
\[ \sigma_\epsilon(A) = \displaystyle \bigcup_{i=1}^n \sigma_\epsilon(A_i).\]

\section{Finite Rank operators} \label{sec: finiterank}

The majority of this paper has focused on both explicit and asymptotic characterizations of $\epsilon$-pseudospectra for various classes of finite dimensional linear operators. A natural next step is to consider finite rank operators on an infinite dimensional space.

In section~\ref{sec: pseudospectra} we defined $\epsilon$-pseudospectra for matrices, although our definitions are exactly the same in the infinite dimensional case. For our purposes, the only noteworthy difference between matrices and operators is that the spectrum of an operator is no longer defined as the collection of eigenvalues, but rather
\[ \sigma(A) = \{ \lambda \mid \lambda I - A \text{ does not have a bounded inverse} \} \]
As a result, we do not get the same properties for pseudospectra as we did previously; in particular, $\sigma_\epsilon(A)$ is not necessarily bounded.

That being said, the following theorem shows that finite rank operators behave similarly to matrices, in that asymptotically the radii of $\epsilon$-pseudospectra are bounded  by powers of epsilon. The following theorem makes this precise. 

\begin{thm} Let $V$ be a Hilbert space and $A: V \to V$ a finite rank operator on $H$. Then there exists $C$ such that for sufficiently small $\epsilon$, 
\[\sigma_\epsilon(A) \subseteq \sigma(A) + B(0, C \epsilon^{\frac{1}{m+1}}).\] 
where $m$ is the rank of $A$.
Furthermore, this bound is sharp in the sense that there exists a rank-$m$ operator $A$ and a constant $c$ such that
\[ \sigma_\epsilon(A) \supseteq \sigma(A) + B(0, c \epsilon^{\frac{1}{m+1}}) \]
for sufficiently small $\epsilon$.
\label{finiterank}
\end{thm}

\begin{proof} Since $A$ has finite rank, there exists a finite dimensional subspace $U$ such that $V=U \oplus W$ and $A(U) \subseteq U$, and $A(W)= \{0\}$. Choosing an orthonormal basis for $A$ which respects this decomposition we can write $A=A' \oplus 0$. Then the spectrum of $A$ is $\sigma(A') \cup \{0\}$, and we know that for any $\epsilon$,
\[\sigma_\epsilon(A) = \sigma_\epsilon(A') \cup  \sigma_\epsilon (0). \] 

The $\epsilon$-pseudospectrum of the zero operator is well-understood since this operator is normal; for any $\epsilon$, it is precisely the ball of radius $\epsilon$. It thus suffices to consider the $\epsilon$-pseudospectrum of the finite rank operator $A':U \to U$, where $U$ is finite dimensional. The  $\epsilon$-pseudospectrum of this operator goes like $\epsilon^{1/j}$, where $j$ is the dimension of the largest Jordan block; we will prove that $j \leq m+1$. Note that the rank of the  $n \times n$ Jordan block given by 
\[A= \begin{pmatrix}
 \lambda &1 & 0 & 0 & \ldots \\
 0 & \lambda & 1 & 0 & \ldots \\
\vdots & \vdots & \ddots & \ddots & \ldots \\
\vdots & \vdots &\vdots & \ddots & b_{n-1} \\
0 & 0 & 0 & 0  & \lambda \\
 \end{pmatrix} \]
is $n$ if $\lambda \neq 0$, and $n-1$ if $\lambda = 0$. Since we know that the rank of $A$ is larger than or equal to the rank of the largest Jordan block, we have  an upper bound on the dimension of the largest Jordan block: it is of size $m+1$, with equality attained when $\lambda = 0$. By Thm. \ref{thm: AUDiT}, we then know that $\sigma_\epsilon(A)$  is contained, for small enough $\epsilon$, in the set $ \sigma(A) + C \epsilon^{\frac{1}{m+1}}$. 

Note that this bound is sharp; we can see this by taking $V$ to be $\R^{m+1}$ and considering the rank-m operator \[A_m= \begin{pmatrix}
 0 &1 & 0 & 0 & \ldots \\
 0 & 0 & 1 & 0 & \ldots \\
\vdots & \vdots & \ddots & \ddots & \ldots \\
\vdots & \vdots &\vdots & \ddots & 1 \\
0 & 0 & 0 & 0  & 0 \\
 \end{pmatrix} \] the pseudospectrum of which will contain the ball of radius $\epsilon^{\frac{1}{m+1}}$ by proposition \ref{jordanblock}. \end{proof}
 
\begin{openqs}
\hspace{1cm} 

The natural question to ask now is whether we can extend this result to more arbitrary operators on Hilbert spaces. In particular, for a bounded operator $A$, we would like to establish if there exists a continuous function $r_A(\epsilon)$ such that for sufficiently small $\epsilon$, 
\[ \sigma_\epsilon(A) \subseteq \sigma(A) + B(0, r_A(\epsilon)). \]
For a matrix $A$, we proved in Thm. \ref{thm: AUDiT} that $r_A(\epsilon) = C\epsilon^{1/n}$, where $n$ is the size of the largest Jordan block associated to $A$, and $C$ is a constant that depends on the left and right eigenvectors associated to a certain eigenvalue. For a finite rank operator $A$, we proved in Thm. \ref{finiterank} that $r_A(\epsilon) = C\epsilon^{\frac{1}{m+1}}$, where $m$ is the rank of the operator and $C$ is as above.

For closed but not necessarily bounded operators, the picture is more complex, as the spectrum need not be bounded or even non-empty.
For example, the operator $A : u \mapsto u'$ in $L^2[0,1]$ with domain $D(A)$ being the set of absolutely continuous functions on $[0,1]$ satisfying $u(1) = 0$ has empty spectrum.
With $D(A)$ being the entire space, then the spectrum of $A$ is the entire complex plane. Davies \cite{davies1999pseudo} also provides an example of an unbounded operator with unbounded pseudospectrum.

Given these examples, we can see that Thm. \ref{finiterank} will not generalize to unbounded operators, as the pseudospectrum of an unbounded operator may be unbounded for all $\epsilon$. 

Nonetheless, we do still have a certain convergence of the $\epsilon$-pseudospectrum to the spectrum  \cite[\S 4]{trefethen2005spectra}, namely $ \cap_{\epsilon > 0} \; \sigma_\epsilon(A) = \sigma(A) $. Also, while the $\epsilon$-pseudospectrum may be unbounded, any bounded component of it necessarily contains a component of the spectrum. These results imply that the bounded components of the $\epsilon$-pseudospectrum must converge to the spectrum. Therefore, if we restrict our attention to these bounded components, we can attempt to generalize Thms. \ref{thm: AUDiT} and \ref{finiterank} by asking whether the \emph{bounded} components of $\sigma_\epsilon(A)$ converge to the spectrum as a union of disks.

\end{openqs}

\section*{Acknowledgements}
Support for this project was provided by the National Science Foundation REU Grant DMS-0850577 and DMS-1347804, the Clare Boothe Luce Program of the Henry Luce Foundation, and the SMALL REU at Williams College.

\bibliographystyle{plain}
\bibliography{PseudospectraReferences}

\end{document}